\documentclass[reqno,oneside,12pt,hidelinks]{amsart}
\usepackage{amssymb,amsmath,amsthm,amsxtra,calc,bm, color}
\usepackage{enumerate}
\usepackage[margin=.95in]{geometry}

\usepackage{afterpage}

\usepackage[scr=boondoxo]{mathalfa}
\usepackage{mathtools}

\usepackage{mleftright}
\mleftright

\usepackage{nccmath}
\usepackage{cases}

\usepackage{calc}
\usepackage{graphicx}
\usepackage{hyperref}
\usepackage[draft]{microtype}

\usepackage{comment}

\def\Z{\mathbb{Z}}
\def\Q{\mathbb{Q}}

\def\H{\mathbb{H}}

\def\C{\mathbb{C}}

\newcommand{\slz}{\operatorname{SL}_2(\mathbb{Z})}
\renewcommand{\pmod}[1]{\hspace{.5em}\left(\operatorname{mod}#1\right)}

\newcommand{\abcd}{\ensuremath{\left(\begin{smallmatrix}a&b\\c&d\end{smallmatrix}\right)}}

\renewcommand{\tilde}{\widetilde}

\def\spt{\mathrm{spt}}

\newtheorem{theorem}{Theorem}[section]
\newtheorem{lemma}[theorem]{Lemma}
\newtheorem{corollary}[theorem]{Corollary}

\theoremstyle{remark}

\numberwithin{equation}{section}

\mathtoolsset{showonlyrefs=true}

\title[Hecke Relations for Eta Multipliers]{Hecke Relations for Eta Multipliers and Congruences of Higher-Order Smallest Parts Functions}

\date{\today}
\author{Clayton Williams}
\address{Department of Mathematics\\
University of Illinois\\
Urbana, IL 61801} 
\email{cw78@illinois.edu} 

\begin{document}

\begin{abstract}
    We derive identities from Hecke operators acting on a family of Eisenstein-eta quotients, yielding congruences for their coefficients modulo powers of primes. As an application we derive systematic congruences for several higher-order smallest parts functions modulo prime powers, resolving a question of Garvan for these cases. We also relate moments of cranks and ranks to the partition function modulo prime powers. Some of our results strengthen and generalize those of a 2023 paper by Wang and Yang.
\end{abstract}

\maketitle


\section{Introduction}

The partition function $ p(n) $ gives the number of ways of writing $ n\geq 0 $ as a sum of positive integers, and is given by the generating function 
\begin{align*}
\sum_{n=0}^\infty p(n)q^n= \frac{q^{1/24}}{\eta(z)}.
\end{align*}
Ramanujan proposed the following congruences for the partition function:
\begin{align}\label{eqn:ram-cong}
p \left(\frac{\ell^{m} n+1}{24} \right) \equiv 0\begin{dcases}\,\left(\operatorname{mod}\ell^{m}\right)&\ell=5,\\
  \,\left(\operatorname{mod}\ell^{\lfloor m\rfloor+1}\right)&\ell=7,\\
  \,\left(\operatorname{mod}\ell^{m}\right)&\ell=11.\end{dcases}
\end{align}
Note that $ p(n)=0 $ if $ n\not\in \Z_{> 0} $. Proofs were given by Ramanujan, Watson, and Atkin \cite{MR1544457}, \cite{MR1581588}, \cite{MR0205958}.\par 

In 2008 Andrews \cite{MR2456627} introduced the smallest parts function $ \spt(n) $  which counts the number of smallest parts in the partitions of $ n. $ There are also Ramanujan-type congruences for $ \spt(n) $. One such congruence was found by Ono \cite[Corollary 1.2]{MR2770948}, later strengthened in \cite[Theorem 1.1]{MR2822242}. In particular, if $ \ell\geq 5 $ is a prime and $ \left(\frac{-n}{\ell} \right)=1 $ then we have 
\begin{gather}\label{eqn: spt-cong}
\spt \left(\frac{\ell^{2m}n+1}{24}\right)\equiv 0\pmod{\ell^m}.
\end{gather}
This congruence was later explained using identities arising from Hecke operators in \cite{MR3228930}.
 
One can define higher-order smallest parts functions $ \spt_j. $ Write $ \pi\vdash n $ if $ \pi $ is a partition of $ n $. Garvan \cite[Definition 5.2]{MR2822233} defined a weight function $ \omega_j $ for partitions and from this introduced the higher-order smallest parts function
\begin{align*}
\spt_j(n):=\sum_{\pi\vdash n}\omega_j(\pi).
\end{align*}
 Garvan shows $ \spt_1(n)=\spt(n) $ and obtains a number of congruences modulo $ 2,3,5,7, $ and $ 11 $ for $ \spt_2,\,\spt_3,$ and $\spt_4. $ He then asks if there are systematic congruences for $ \spt_j $ modulo primes $ \ell\geq 5 $. We answer this in the affirmative for the first $ 4 $ higher-order smallest parts functions.

\begin{theorem}\label{thm: spt-main}
Let $ \delta_{j,\ell} $ be the Kronecker delta symbol. If $ \ell\geq 5 $ is prime and $ \left(\frac{-n}{\ell} \right)=1 $ we have
\begin{gather}
\spt_2 \left(\frac{\ell^{2m}n+1}{24} \right)\equiv 0\pmod{\ell^{m}},\\
\spt_3 \left(\frac{\ell^{2m}n+1}{24} \right)\equiv 0\pmod{\ell^{m-\delta_{5,\ell}}},\\
\spt_4 \left(\frac{\ell^{2m}n+1}{24} \right)\equiv 0\pmod{\ell^{m+\delta_{5,\ell}-\delta_{7,\ell}}},\\
\spt_5 \left(\frac{\ell^{2m}n+1}{24} \right)\equiv 0\pmod{\ell^{m+\delta_{5,\ell}+\delta_{7,\ell}}}.
\end{gather}
  
\end{theorem}

This is sharp in at least some cases; for example,
\begin{gather*}
\spt_3 \left(\frac{5^2 \cdot 119+1}{24} \right)\equiv 1\pmod{5},\\
\spt_4 \left(\frac{7^2 \cdot 47+1}{24} \right)\equiv 4\pmod{7}.
\end{gather*}
In order to prove this theorem we use Hecke operators acting on spaces of half-integral weight weakly holomorphic modular forms, strengthening and generalizing some of the results of \cite{MR4490457}. Recall that a function $ f:\H\to\C $ is a weakly holomorphic modular form of level $ 1 $, weight $ \lambda+1/2 $, and multiplier $ \nu $ if it is holomorphic on the complex upper half plane $ \H $, its poles are supported at the cusp $ \infty $, and it satisfies
\begin{align*}
  f \left(\gamma z\right) = \nu(\gamma)(cz+d)^{\lambda+1/2}f(z)
  \end{align*}
  for all $ \gamma=\abcd\in \slz. $ Here $ \abcd z = \frac{az+b}{cz+d}  $ is the usual M\"{o}bius fractional linear transformation. We denote the vector space of weakly holomorphic modular forms of level $ 1 $ and multiplier $ \nu $ by $ M^!_{\lambda+1/2}(1,\nu) $ or $ M^!_{\lambda+1/2}(\nu). $\par 

A particular family of multipliers is given by $ \nu_\eta^{-s}, $ where $ (s,2)=1 $, and $ \nu_\eta $ is the $ \eta $ multiplier defined, for example, in \cite[Theorem 5.8.1]{MR3675870}. By definition $ \nu_\eta(\gamma)^{24}=1 $ for all $ \gamma\in \slz. $ When $ S_{\lambda+1/2}(\nu_\eta^{-s})=\{0\}$, modular forms in $ M^!_{\lambda+1/2}(\nu_\eta^{-s}) $ are uniquely determined by their principal parts. \par 

Define
\begin{align*}
Z := \left\{(k,s): k\in \{0,4,6,8,10,14\},\,0<s<24,\,(s,2)=1\right\}.
\end{align*}
Let $ E_k $ be the Eisenstein series of weight $ k, $ with $ E_0(z):=1. $ If $ (k,s)\in Z $ then $ S_{k-s/2}(\nu_\eta^{-s}) =\{0\}.$ Define $ a_{k,s}(n) $ for $ (k,s )\in Z$ by 
\begin{align}
g_{k,s}(z):=\frac{E_k(z)}{\eta(z)^{s}}=\sum_{n\equiv -s\pmod{24}} a_{k,s}(n)q^{n/24}\in M^!_{k-s/2}(\nu_\eta^{-s})\cap \Z(\!(q^{1/24})\!).
\end{align}
Here $ q=e^{2\pi i z} $ for $ z\in \H. $ The next two theorems include some results of \cite{MR4490457} as special cases; see Corollary \ref{cor:wy-analogue} below. The following result is a specialization of Theorem \ref{thm: general-coeff-cong}.

\begin{theorem}\label{thm: main-ak-cong}
  Let $ (k,s)\in Z $ and $ \ell\geq 5 $ be prime. If $ 2k>s+2 $ then for $ m\geq 1 $ we have 
    \begin{gather}\label{eqn: coeff-cong}
  a_{k,s}(\ell^2n)\equiv \ell^{k-\frac{s+3}{2} }\left(\frac{12}{\ell} \right)\left(\frac{-1}{\ell} \right)^{\frac{s+1}{2} }\left[\left(\frac{-s}{\ell} \right)-\left(\frac{n}{\ell} \right)\right]a_{k,s}(n)\pmod{\ell^{2k-s-2}},\\
    a_{k,s}(\ell^{2m+2}n)\equiv \ell^{k-\frac{s+3}{2} }\left(\frac{-1}{\ell} \right)^{\frac{s+1}{2} }\left(\frac{-12s}{\ell} \right)a_{k,s}(\ell^{2m}n)\pmod{\ell^{(2k-s-2)m}}.\label{eqn: coeff-cong2}
  \end{gather}
  Moreover, in the case $ (s,6)=3 $, if $ 3\mid n $ then 
  \begin{gather}\label{eqn: coeff-cong3}
    a_{k,s}(9n)\equiv 3^{k-\frac{s+3}{2} }\left(\frac{-1}{3} \right)^{\frac{s+1}{2} }\left[\left(\frac{-s/3}{3} \right)-\left(\frac{n/3}{3} \right)\right]a_{k,s}(n)\pmod{3^{2k-s-2}}.
  \end{gather}
  If $ (s,6)=3 $ and $ m\geq 1 $ then
  \begin{gather} 
    a_{k,s}(3^{2m+2}n)\equiv 3^{k-\frac{s+3}{2} }\left(\frac{-1}{3} \right)^{\frac{s+1}{2} }\left(\frac{-s/3}{3} \right)a_{k,s}(3^{2m}n)\pmod{3^{(2k-s-2)m}}.
  \end{gather}
  
  \end{theorem}

  For $ (k,s)\in Z $ and $ D\equiv s\pmod{24}, $ $ D>0 $, define the unique weakly holomorphic modular form of weight $ k-s/2 $ with principal part $ q^{-D/24} $:  
\begin{align}
f_{D,k,s}(z):=\sum_{n\equiv -s\pmod{24}}b_{D,k,s}(n)q^{n/24}=q^{-D/24}+O(1)\in M_{k-s/2}^!(\nu_\eta^{-s})\cap \Z(\!(q^{1/24})\!).
\end{align}
Note, in particular, that $ g_{k,s} = f_{s,k,s} $.
\begin{theorem}\label{thm: main-coeff-cong-ak-bd}
  Let $ (k,s)\in Z $ and $ \ell\geq 5 $ be prime. Let $ m\geq 0$. If either $ \left(\frac{-ns}{\ell} \right)=1 $, or $ \ell\mid\mid n $ and $ \ell\mid\mid s$, then
  \begin{align}
  a_{k,s}(\ell^{2m}n)=\ell^{(2k-s-2)m}b_{\ell^{2m}s,k,s}(n).
  \end{align}
  
  \end{theorem}

We obtain the following congruences for the coefficients $ a_{k,s}(n) $ from Theorems \ref{thm: main-ak-cong} and \ref{thm: main-coeff-cong-ak-bd}. 
  \begin{corollary}\label{cor:wy-analogue}
  Let $ (k,s)\in Z $ and $ \ell\geq 5 $ be prime. Let $ m\geq 0 $. If $ 2k>s+3 $ then we have
  \begin{gather}\label{eqn:wy-analogue-1.1}
  a_{k,s}(\ell^{2m}n)\equiv 0\pmod{\ell^{k-\frac{s+3}{2}}}.
  \end{gather}
  Let $ 2k>s-2 .$ If either $ \left(\frac{-ns}{\ell} \right)=1 $, or $ \ell\mid\mid n $ and $ \ell\mid\mid s $ then we have
  \begin{gather}\label{eqn:wy-analogue-2.8i}
  a_{k,s}(\ell^{2m}n)\equiv 0\pmod{\ell^{(2k-s-2)m}}.
  \end{gather}
  If $ \ell\nmid n $ and $ m\geq 1 $ we have 
  \begin{gather}\label{eqn:wy-analogue-2.8ii-1}
  a_{k,s}(\ell^{2m+1} n)\equiv \ell^{k-\frac{s+3}{2}}\left(\frac{-1}{\ell} \right)^{\frac{s+1}{2}}\left(\frac{-12D}{\ell} \right)a_{k,s}(\ell^{2m-1} n)\pmod{\ell^{(2k-s-2)m}}.
\end{gather}
      
  \end{corollary}
  
As in Theorem \ref{thm: main-ak-cong}, there are analagous results for $ \ell=3 $ when $ (s,6)=3. $
Equation \eqref{eqn:wy-analogue-1.1} gives \cite[Theorem 1.1]{MR4490457} in the special case $ s=1; $ \eqref{eqn:wy-analogue-2.8i} gives \cite[Theorem 2.8 (1)]{MR4490457} when $ s=1 $ and $ m=1,2 $; and \eqref{eqn:wy-analogue-2.8ii-1} gives \cite[Theorem 2.8 (2)]{MR4490457} when $ s=1 $ and $ m=1,2 $ (correcting a typographical error in that statement). Note that $ a_{k,1}  $ corresponds to $ e_{k} $ in the notation of \cite{MR4490457}, as $ a_{k,1}(n)=e_k \left(\frac{n+1}{24} \right) $ for $ k\in \left\{4,6,8,10,14\right\}. $\par 

Our aim throughout is to obtain results uniform in $ \ell; $ consequently some congruences may be improved for small primes. Computational evidence suggests the congruence \eqref{eqn:wy-analogue-2.8i} may be suboptimal for small $ \ell$ for some $ k,s $. There are examples showing the congruence is sharp for larger $ \ell; $ for instance $ a_{4,1}(29^2\cdot 23)\not\equiv 0\pmod{29^6}. $ \par   

Theorems \ref{thm: main-ak-cong} and \ref{thm: main-coeff-cong-ak-bd} allow us to prove congruences for partition statistics including the higher-order smallest parts functions. An alternative combinatorial proof of the Ramanujan congruences \eqref{eqn:ram-cong} modulo $ 5 $ and $ 7 $ was given by Dyson \cite{MR3077150}, wherein the concept of the rank of a partition was introduced. The rank of a partition $ \pi\vdash n $ is $ \max\pi-\ell(\pi) $, where $ \ell(\pi) $ is the number of parts of $ \pi $. Dyson also conjectured the existence of a partition statistic he called the crank, which would prove the congruences modulo $ 11 $ combinatorially. This statistic was found by Andrews and Garvan \cite{MR0929094}.\par 

Let $ N(m,n) $ be the number of partitions of $ n $ with rank $ m $ and $ M(m,n) $ be the number of partitions of $ n $ with crank $ m. $ Define the $ j $th moment of ranks and cranks, respectively, by
\begin{align}
N_j(n):=\sum_k k^j N(k,n),\\
M_j(n):=\sum_k k^j M(k,n).
\end{align}
The only nonzero moments occur for even $ j $; this is because $ N(k,n)=-N(-k,n) $ and similarly for $ M. $ We will also need the symmetrized moments of ranks and cranks. These are linear combinations of the $ N_{2j} $ and $ M_{2j} $; the $ 2j $th symmetrized crank and rank moments are, respectively,
\begin{gather}\label{eqn:mu-defn}
  \mu_{2j}(n):=\sum_{m}\binom{m+\lfloor \frac{2j-1}{2} \rfloor}{2j}M(m,n),\\
  \eta_{2j}(n):=\sum_{m}\binom{m+\lfloor \frac{2j-1}{2} \rfloor}{2j}N(m,n).\label{eqn:eta-defn}
\end{gather}
Andrews \cite{MR2308850} introduced the $ \eta_{2j} $ function. Garvan later introduced the $ \mu_{2j} $ and $ \spt_j $ functions and proved \cite[Theorem 5.6]{MR2822233} that
\begin{align}\label{eqn:spt-defn}
\spt_j(n)=\mu_{2j}(n)-\eta_{2j}(n).
\end{align}

In \cite[(4.8)]{MR2035811} and \cite[Lemma 8.5]{MR4490457} it was shown that the generating functions for $ M_{2j},\, N_{2j} $ can be written as linear combinations of modular forms and their derivatives. From this Wang and Yang found identities such as
\begin{gather*}
M_4(n) = \frac{1}{20}a_{4,1}(24n-1)-\frac{1}{20} p(n)+2np(n)-12n^2p(n),\\
N_4(n) = \frac{2}{15} a_{4,1}(24n-1)-\frac{2}{15} p(n) + 4np(n) -36n^2p(n)+N_2(n)-12nN_2(n).
\end{gather*}
See \cite[(8.24),(8.25),(8.29),(8.30)]{MR4490457}.
From these identities, together with Theorems \ref{thm: main-ak-cong} and \ref{thm: main-coeff-cong-ak-bd}, we can prove the following relations for moments of cranks.

\begin{theorem}\label{thm: crank-congruences}
For $ \ell\geq 5 $ prime and $ \left(\frac{-n}{\ell} \right)=1 $ we have
\begin{gather}
240M_4 \left(\frac{\ell^{2m}n+1}{24} \right) \equiv \left(3+10\ell^{2m}n-5\ell^{4m}n^2\right)p \left(\frac{\ell^{2m}n+1}{24} \right)\pmod{\ell^{5m}},\label{eqn:cong-M4}\\
12096 M_6 \left(\frac{\ell^{2m}n+1}{24} \right)\equiv\left(27+189\ell^{2m}n-315\ell^{4m}n^2 \right)p \left(\frac{\ell^{2m}n+1}{24} \right)\pmod{\ell^{5m}},\label{eqn:cong-M6}\\
5760\mu_4\left(\frac{\ell^{2m}n+1}{24} \right)\equiv \left(-17-10\ell^{2m}n-5\ell^{4m}n^2\right)p\left(\frac{\ell^{2m}n+1}{24} \right)\pmod{\ell^{5m}},\label{eqn:cong-mu4}\\
967680\mu_6 \left(\frac{\ell^{2m}n+1}{24} \right) \equiv (367+189\ell^{2m}n+105\ell^{4m}n^2)p\left(\frac{\ell^{2m}n+1}{24} \right)\pmod{\ell^{5m}}.\label{eqn: cong-mu6}
\end{gather}   
\end{theorem}

Similar congruences modulo $ \ell^{2m} $ can be deduced from \cite[\S 8]{MR4490457}.\par

We also obtain relations for moments of ranks. These require an additional identity for the $ N_2 $ function, which we obtain using a result of Ahlgren and Kim \cite[Cor.~3.2]{MR3228930}. These results imply, for example, that for $ \ell\geq5 $ prime, $ m\geq 1 $, and $ \left(\frac{-n}{\ell} \right)=1 $, we have
\begin{align}\label{eqn: cong-rank-moments}
12N_2 \left(\frac{\ell^{2m}n+1}{24} \right)\equiv 80 N_4\left(\frac{\ell^{2m}n+1}{24} \right)\equiv 448 N_6\left(\frac{\ell^{2m}n+1}{24} \right)\equiv p\left(\frac{\ell^{2m}n+1}{24} \right)\pmod{\ell^m},
\end{align}
see Theorem \ref{thm:rank-congruences}. Congruences for the $ 8 $th and $ 10 $th moments of ranks and cranks can also be found using our method.\par

The outline of this paper is as follows: in Section~\ref{sec:hecke-reln-eis-eta} we prove generalizations of Theorems~\ref{thm: main-ak-cong} and \ref{thm: main-coeff-cong-ak-bd} by applying Hecke operators to the functions $ f_{D,k,s} $ and identifying principal parts. In Section~\ref{sec: crank-moment} we apply the congruences of $ a_{k,1} $ to the identities of crank moments to prove Theorem \ref{thm: crank-congruences}. In Section~\ref{sec: rank-moment} we derive a congruence for $ N_2 $ and apply this, together with the congruences for $ a_{k,1}, $ to derive systematic congruences for moments of ranks. Finally, in Section~\ref{sec: spt-proof} we prove Theorem \ref{thm: spt-main}.

\subsection*{Acknowledgments}
The author thanks Scott Ahlgren for many helpful discussions in the preparation of this paper.

\section{Hecke Relations for Eta Multipliers}\label{sec:hecke-reln-eis-eta}
In \cite{MR2881327} a simple general formula is proved for the action of Hecke operators on the spaces of weight $ \frac{1}{2}  $ modular forms of level $ 4. $ This result explains many congruences for the traces of singular moduli. Here we obtain a similar result for the spaces $ M_{k-s/2}^!(1,\nu_\eta^{-s}) $ for $ (k,s)\in Z. $\par

Let $ (s,2) =1$. Note for $ 0<s<24 $ that $ f\in S_{k-s/2}(1,\nu_\eta^{-s}) $ only if $ \eta^s f\in S_{k}(1) $. Since for $ (k,s)\in Z $ we have $ S_{k}(1)=\{0\}$, see that $ S_{k-s/2}(1,\nu_\eta^{-s})=\{0\} $ for $ (k,s)\in Z. $\par   

Define the principal part of a $ q $-series as
\begin{align*}
\operatorname{p.}\sum_{n\gg -\infty} a(n)q^{n/24} := \sum_{n\leq 0}a(n)q^{n/24}.
\end{align*}
Because $ S_{k-s/2}(\nu_\eta^{-s})=\{0\} $ for $ (k,s)\in Z $ we see that weakly holomorphic modular forms in $ M^!_{k-s/2}(\nu_\eta^{-s}) $ are uniquely determined by their principal parts.

When $ (s,2) =1$ and for primes $ \ell\geq 5 $ we have a Hecke operator $ T(\ell^2):M_{k-s/2}^!(\nu_\eta^{-s})\to M_{k-s/2}^!(\nu_\eta^{-s}) $; its action on coefficients is given by
\begin{align}\label{eqn: hecke-op-defn}
\left(\sum a(n)q^{n/24}\right)\big|T(\ell^2):=\sum \left(a(\ell^2n)+\ell^{k-\frac{s+3}{2} }\left(\frac{-1}{\ell} \right)^{\frac{s+1}{2} }\left(\frac{12n}{\ell} \right)a(n)+\ell^{2k-s-2}a \left(\frac{n}{\ell^2} \right)\right)q^{n/24}.
\end{align}
See, for example, \cite[(3.16)]{AAD2023shimura}, \cite[(2.2)]{MR3228930}, or \cite[Proposition 11]{MR3263525}. The algebra of Hecke operators is commutative. For prime powers the Hecke operator $ T(\ell^{2m}) $, $ m\geq 1, $ is defined recursively by
\begin{align}\label{eqn:recursive-hecke}
  T(\ell^{2m+2}) := T(\ell^{2})T(\ell^{2m})-\ell^{2k-s-2}T(\ell^{2m-2}).
  \end{align}
 \par
 When $ (s,6)=3 $ we define $ T(9) $ on $ M_{k-s/2}^!(\nu_\eta^{-s}) $ \cite[(3.17)]{AAD2023shimura}
\begin{multline}\label{eqn: hecke-9}
\left(\sum a(n)q^{n/24}\right)\big|T(9)  :=\sum\left(a(9n)+3^{k-\frac{s+3}{2} }\left(\frac{-1}{3} \right)^{\frac{s+1}{2} }\left(\frac{n/3}{3} \right)a(n)+3^{2k-s-2}a \left(\frac{n}{9} \right)\right)q^{n/24}.
\end{multline}
Recall that coefficients of modular forms with multiplier $ \nu_\eta^{-s} $ are supported on $ n\equiv -s\pmod{24} $. In most of the following claims we state and prove only the case $ \ell\geq 5, $ however, the same methods extend to $ \ell =3 $ using \eqref{eqn: hecke-9}.\par

For $ (k,s)\in Z $ and $0<D\equiv s\pmod{24} $ let
\begin{align*}
f_{D,k,s}(z) =q^{-D/24}+O(1) = \sum_{n\equiv-s\pmod{24}}b_{D,k,s}(n)q^{n/24}
\end{align*}
be the unique modular form in $ M_{k-s/2}^!(1,\nu_\eta^{-s}) $ with principal part equal to $ q^{-D/24}. $ For $ \ell\geq 5 $ a prime with $ \ell^2\nmid D $ we define \[ F_{D,k,s,\ell}^{(m)}(z)=\sum_{n\equiv-s\pmod{24}}c_{D,k,s,\ell}^{(m)}(n)q^{n/24}\in M_{k-s/2}^!(\nu_\eta^{-s}) \] for $ m\geq 0 $ by
\begin{gather}
  F_{D,k,s,\ell}^{(0)}(z):=f_{D,k,s}(z),\\
  F_{D,k,s,\ell}^{(m)}(z)
  :=f_{D,k,s}(z)\Big| \left(T(\ell^{2m})-\ell^{k-\frac{s+3}{2}}\left(\frac{-1}{\ell} \right)^{\frac{s+1}{2} }\left(\frac{-12D}{\ell} \right)T(\ell^{2m-2})\right),\hspace{1em}m\geq 1.\label{eqn: FDm}
\end{gather}

The foundational result of this section is the following analogue of \cite[Theorem 2]{MR2881327}.

\begin{theorem}\label{thm:tech-main}
 Let $ (k,s)\in Z $, $ 0<D\equiv s\pmod{24} $, and $  \ell\geq5 $ be prime with $ \ell^2\nmid D $. Then for all $ n $ we have
\begin{align}\label{eqn: cDm}
F_{D,k,s,\ell}^{(m)}(z)=\ell^{(2k-s-2)m}f_{\ell^{2m}D,k,s}(z).
\end{align}
    
  \end{theorem}

\begin{proof}
Fix $ D,k,s,\ell$ satisfying the hypotheses above and for brevity write
$
f_D=f_{D,k,s}
$, $ {b_D=b_{D,k,s},} $ $ F_D^{(m)}=F_{D,k,s,\ell}^{(m)}$, and $ c_D^{(m)}=c_{D,k,s,\ell}^{(m)} $. The $ F_D^{(m)} $ have a recurrence relation for $ m\geq 2 $ given by
\begin{align}\label{eqn: FD-recursion}
F_D^{(m)}=F_D^{(m-1)}\big|T(\ell^{2})-\ell^{2k-s-2}F_D^{(m-2)}.
\end{align}
  This is obtained by expanding the $ T(\ell^{2m}),\,T(\ell^{2m-2}) $ operators using \eqref{eqn:recursive-hecke} and the commutativity of the Hecke operators to obtain
  \begin{align*}
  F_D^{(m)}=f_D \big|&\left(T(\ell^{2m-2})-\ell^{k-\frac{s+3}{2} }\left(\frac{-1}{\ell} \right)^{\frac{s+1}{2}}\left(\frac{-12D}{\ell} \right)T(\ell^{2m-4})\right)\big| T(\ell^2)\\
  &-\ell^{2k-s-2}f_D \big|\left(T(\ell^{2m-4})-\ell^{k-\frac{s+3}{2} }\left(\frac{-1}{\ell} \right)^{\frac{s+1}{2}}\left(\frac{-12D}{\ell} \right)T(\ell^{2m-6})\right).
  \end{align*}
 Recall $ \operatorname{p.}F_D^{(0)} =q^{-D/24} $. Note for $ n<0 $ that $ b_D(\ell^2n)=0 $. For all $ n<0 $ we also have $ b_D(n)\neq 0 $ if and only if $ n=-D, $ and $ b_D \left(\frac{n}{\ell^2} \right)\neq 0 $ if and only if $ n=-\ell^2D. $ Hence $\operatorname{p.}F_D^{(1)}=\ell^{2k-s-2}q^{-\ell^2D/24},$ and $ F_D^{(1)} = \ell^{2k-s-2}f_{\ell^2 D}. $\par

We proceed by induction. The base case $ F_D^{(0)} = f_{D} $ is vacuously true, and we just showed $ F_D^{(1)} = \ell^{2k-s-2}f_{\ell^2 D}. $  Suppose inductively, therefore, that for all $ m'\leq m $ we have
  \begin{align*}
  F_D^{(m')} = \ell^{(2k-s-2)m'}f_{\ell^{2m'}D}.
  \end{align*}
From \eqref{eqn: hecke-op-defn} and \eqref{eqn: FD-recursion} we obtain $ F_D^{(m)} = \ell^{(2k-s-2)m}f_{\ell^{2m}D,k,s} $ for all $ m\geq 1 $ by induction, proving the theorem. 
\end{proof}

These $ c_{D,k,s,\ell}^{(m)} $ coefficients can be related to the $ b_{D,k,s} $ coefficients for the same $ D $.

\begin{lemma}\label{lem:tech-1}
  With the same hypotheses as Theorem \ref{thm:tech-main} and $ m\geq 1 $ we have 
\begin{multline}
  \label{eqn:cDm-bD-reln}
c_{D,k,s,\ell}^{(m)}(\ell^{2}n)-\ell^{(2k-s-2)}c_{D,k,s,\ell}^{(m-1)}(n) = b_{D,k,s}(\ell^{2m+2}n)-\ell^{k-\frac{s+3}{2} }\left(\frac{-1}{\ell} \right)^{\frac{s+1}{2} }\left(\frac{-12D}{\ell} \right)b_{D,k,s}(\ell^{2m}n).
\end{multline}
  
\end{lemma}

\begin{proof}
Fix $ D,k,s,\ell $ and write $ c^{(m)}_D=c^{(m)}_{D,k,s,\ell} $, $ b_D $ similarly. When $ m=1 $ equation \eqref{eqn: FDm} yields
\begin{multline}\label{eqn:cD1-bD-reln}
  c_{D}^{(1)}(n)=b_D(\ell^2n)+\ell^{k-\frac{s+3}{2}}\left(\frac{-1}{\ell} \right)^{\frac{s+1}{2}}\left(\frac{12}{\ell} \right)\left[\left(\frac{n}{\ell} \right)-\left(\frac{-D}{\ell} \right)\right]b_D(n)+\ell^{2k-s-2}b_D \left(\frac{n}{\ell^2} \right).
\end{multline}
 Substituting $ n\mapsto\ell^2n $ yields the base case because $ c_D^{(0)}=b_D $. For $ m>1 $ equation \eqref{eqn: FD-recursion} yields
\begin{align}\label{eqn:cDm-bD-reln-prime}
c_D^{(m)}(n) &= c_D^{(m-1)}(\ell^2n)+\ell^{k-\frac{s+3}{2}}\left(\frac{-1}{\ell} \right)^{\frac{s+1}{2}}\left(\frac{12n}{\ell} \right)c_D^{(m-1)}(n)\\
&\hspace{3em}+\ell^{2k-s-2}c_D^{(m-1)}\left(\frac{n}{\ell^2} \right)-\ell^{2k-s-2}c_D^{(m-2)}(n).
\end{align}
Substituting $ n\mapsto\ell^2n $ again yields 
 \begin{align*}
 c^{(m)}_D(\ell^2n)-\ell^{2k-s-2}c^{(m-1)}_D(n) = c_D^{(m-1)}(\ell^4n)-\ell^{2k-s-2}c^{(m-2)}_D(\ell^2n).
 \end{align*}
  The result follows inductively.
\end{proof}

Theorem \ref{thm: main-ak-cong} follows from the next theorem when $ D=s. $

\begin{theorem}\label{thm: general-coeff-cong}
  
Let $ (k,s)\in Z $, $ 0<D\equiv s\pmod{24} $, and $ \ell\geq 5 $ be prime with $ \ell^2\nmid D. $ Then for $ 2k>s+2 $ and $ m\geq 1 $ we have 
  \begin{gather}\label{eqn: bD-recur-l2}
b_{D,k,s}(\ell^2n)\equiv \ell^{k-\frac{s+3}{2} }\left(\frac{12}{\ell} \right)\left(\frac{-1}{\ell} \right)^{\frac{s+1}{2} }\left[\left(\frac{-D}{\ell} \right)-\left(\frac{n}{\ell} \right)\right]b_{D,k,s}(n)\pmod{\ell^{2k-s-2}},\\
  b_{D,k,s}(\ell^{2m+2}n)\equiv \ell^{k-\frac{s+3}{2} }\left(\frac{-1}{\ell} \right)^{\frac{s+1}{2} }\left(\frac{-12D}{\ell} \right)b_{D,k,s}(\ell^{2m}n)\pmod{\ell^{(2k-s-2)m}}.\label{eqn: bd-recur-l2m}
\end{gather}
Moreover, in the case $ (s,6)=3 $ and $ 9\nmid D $, if $ 3\mid n $  then we have 
\begin{gather}\label{eqn: bD-recur-3}
  b_{D,k,s}(9n)\equiv 3^{k-\frac{s+3}{2} }\left(\frac{-1}{3} \right)^{\frac{s+1}{2} }\left[\left(\frac{-D/3}{3} \right)-\left(\frac{n/3}{3} \right)\right]b_{D,k,s}(n)\pmod{3^{2k-s-2}}.\end{gather}
  If $ (s,6)=3 $ and $ 9\nmid D $ then for $ m\geq 1 $ we have 
  \begin{gather}  
  \label{eqn: bD-recur-3-2m}
  b_{D,k,s}(3^{2m+2}n)\equiv 3^{k-\frac{s+3}{2} }\left(\frac{-1}{3} \right)^{\frac{s+1}{2} }\left(\frac{-D/3}{3} \right)b_{D,k,s}(3^{2m}n)\pmod{3^{(2k-s-2)m}}.
\end{gather}

\end{theorem}

\begin{proof}
Note that $ c_{D,k,s,\ell}^{(m)}(n)\equiv 0\pmod{\ell^{(2k-s-2)m}} $ from \eqref{eqn: cDm}. Then \eqref{eqn:cD1-bD-reln} gives \eqref{eqn: bD-recur-l2} while \eqref{eqn: bd-recur-l2m} follows from \eqref{eqn:cDm-bD-reln}. The case $ \ell=3 $ for \eqref{eqn: bD-recur-3} and \eqref{eqn: bD-recur-3-2m} can be handled by proving analogues of Theorem \ref{thm:tech-main} and Lemma \ref{lem:tech-1} using \eqref{eqn: hecke-9}.    
\end{proof}

Now, observe that when $ m>1 $ and $ \ell\geq 5 $ is a prime with $ \ell\mid\mid n $ we have from \eqref{eqn:cDm-bD-reln-prime} \[c_{D,k,s,\ell}^{(m)}(n)=c_{D,k,s,\ell}^{(m-1)}(\ell^2 n)-\ell^{2k-s-2}c_{D,k,s,\ell}^{(m-2)}(n).\]
This, together with \eqref{eqn:cDm-bD-reln} and, in the case $ m=1 $, \eqref{eqn:cD1-bD-reln}, yields the following.

 \begin{lemma}\label{lem:tech-2}
  Let $ (k,s)\in Z $, $ 0<D\equiv s\pmod{24} $, and $ \ell\geq 5 $ with $ \ell^2\nmid D. $ Then for $ m\geq 1 $, if $ \ell\mid\mid n $ we have
 \begin{align*}
 c_{D,k,s,\ell}^{(m)}(n) = b_{D,k,s}(\ell^{2m}n)-\ell^{k-\frac{s+3}{2} }\left(\frac{-1}{\ell} \right)^{\frac{s+1}{2} }\left(\frac{-12D}{\ell} \right)b_{D,k,s}(\ell^{2m-2}n).
 \end{align*}
 
 \end{lemma}
 From this we can also obtain congruences for odd powers of $ \ell $ in the argument.

 \begin{corollary}
   Let $ (k,s)\in Z $, $ 0<D\equiv s\pmod{24} $, and $ \ell\geq 5 $ be prime with $ \ell^2\nmid D. $ Then for $ 2k>s+2 $ and $ m\geq 1 $, if $ \ell\nmid n $ then
 \begin{gather}
 b_{D,k,s}(\ell^{2m+1}n)\equiv \ell^{k-\frac{s+3}{2}}\left(\frac{-1}{\ell} \right)^{\frac{s+1}{2}}\left(\frac{-12D}{\ell} \right)b_{D,k,s}(\ell^{2m-1}n)\pmod{\ell^{(2k-s-2)m}}.
 \end{gather}
  
 \end{corollary}
 
 The proof follows from Theorem \ref{thm:tech-main} and Lemma \ref{lem:tech-2} on substitution of $ n\mapsto \ell n.$ Later we will also require the case $ \ell\nmid n. $
\begin{lemma}\label{lem:tech-3}
  Let $ (k,s)\in Z $, $ 0<D\equiv s\pmod{24} $, and $ \ell\geq 5 $ be prime with $ \ell^2\nmid D. $ Then for $ m\geq 1 $, if $ \ell\nmid n $ then
\begin{align*}
c_{D,k,s,\ell}^{(m)}(n) = b_{D,k,s}(\ell^{2m}n)+\left[1-\left(\frac{-Dn}{\ell} \right)\right]\sum_{j=1}^{m}\left(\ell^{k-\frac{s+3}{2} }\left(\frac{-1}{\ell} \right)^{\frac{s+1}{2} }\left(\frac{12n}{\ell} \right)\right)^j b_{D,k,s}(\ell^{2m-2j}n).
\end{align*}
\end{lemma}

\begin{proof}
We proceed by induction. Fix $ D,k,s,\ell $ and write $ c_D^{(m)},b_D $ as before. When $ m=1 $ equation \eqref{eqn:cD1-bD-reln} establishes the base case. Fix $ m\geq1 $ and suppose that the conclusion holds for $ m. $ Then \eqref{eqn: FD-recursion} gives
\begin{align*}
c_D^{(m+1)}(n) = c_D^{(m)}(\ell^2n)+\ell^{k-\frac{s+3}{2} }\left(\frac{-1}{\ell} \right)^{\frac{s+1}{2} }\left(\frac{12n}{\ell} \right)c_D^{(m)}(n)-\ell^{2k-s-2}c_D^{(m-1)}(n).
\end{align*}
Using Lemma \ref{lem:tech-1}, together with the inductive hypothesis, yields
\begin{align*}
c_D^{(m+1)}(n)& = b_D(\ell^{2m+2}n)+\ell^{k-\frac{s+3}{2} }\left(\frac{-1}{\ell} \right)^{\frac{s+1}{2} }\left(\frac{12n}{\ell} \right)\left[1-\left(\frac{-Dn}{\ell} \right)\right]b_D(\ell^{2m}n)\\
&+\left[1-\left(\frac{-Dn}{\ell}\right)\right]\sum_{j=1}^{m}\left(\ell^{k-\frac{s+3}{2} }\left(\frac{-1}{\ell} \right)^{\frac{s+1}{2} }\left(\frac{12n}{\ell} \right)\right)^{j+1} b_D(\ell^{2m-2j}n).
\end{align*}
The result follows.
\end{proof}

Together these lemmas imply the following result, which gives Theorem \ref{thm: main-coeff-cong-ak-bd} in the case $ D=s. $ 
\begin{theorem}
Let $ (k,s)\in Z $ and $ \ell\geq 5 $ be prime. For $ m\geq 1, $ if either
 $\left(\frac{-nD}{\ell} \right)=1$, or
$ \ell\mid\mid n $ and $ \ell \mid\mid D $,
then
\begin{align*}
b_{D,k,s}(\ell^{2m}n)=\ell^{(2k-s-2)m}b_{\ell^{2m}D,k,s}(n).
\end{align*}

\end{theorem}

\begin{proof}
Suppose first that $ \ell\mid\mid n $ and $ \ell \mid\mid D $. Then Lemma \ref{lem:tech-2}, together with \eqref{eqn: cDm}, gives $ b_{D,k,s}(\ell^{2m}n)=\ell^{(2k-s-2)m}b_{\ell^{2m}D,k,s}(n). $ If $ \left(\frac{-nD}{\ell} \right)=1, $ then the result follows from Lemma \ref{lem:tech-3} and \eqref{eqn: cDm}.
\end{proof}

\section{Consquences for Partition Statistics}

\subsection{4th and 6th Moments of Cranks}\label{sec: crank-moment}

Let $ C_{2j} $ be the generating function for the $ 2j $th moment of cranks, so that
\begin{align*}
C_{2j}(z):=\sum M_{2j}(n)q^{n}.
\end{align*}
Atkin and Garvan \cite[(4.8)]{MR2035811} showed that $ q^{-1/24}\eta(z)C_{2j}(z) $ is a quasimodular form, and so is in the graded ring $ \mathbb{C}[E_2,E_4,E_6] $. Using this, Wang and Yang \cite[(8.24)-(8.25)]{MR4490457} showed that
\begin{align*}
  M_4(n) = \frac{1}{20}a_{4,1}(24n-1)-\frac{1}{20} p(n)+2np(n)-12n^2p(n),
\end{align*}
\vspace{-2em}
\begin{multline*}
  M_6(n)=-\frac{11}{378}a_{6,1}(24n-1)+\frac{1}{14}  a_{4,1}(24n-1)-\frac{3}{14} na_{4,1}(24n-1)\\
  -\frac{8}{189} p(n)+\frac{11}{6} np(n)
  -20n^2p(n)+40 n^3 p(n).
\end{multline*}
They also obtain similar formulas for the symmetrized moments of cranks $ \mu_{2j} $ \cite[(8.37) and (8.38)]{MR4490457}. \par

These identities and Theorem \ref{thm: main-ak-cong} or \cite[Theorem 1.1]{MR4490457} give, for $ \ell\geq 5$ a prime: 
\begin{gather*}
a_{4,1}(\ell^{2m}n)\equiv p \left(\frac{\ell^{2m}n+1}{24} \right)\pmod{20},\\
80M_4\left(\frac{\ell^{2m}n+1}{24} \right)\equiv p\left(\frac{\ell^{2m}n+1}{24} \right)\pmod{\ell^{2m}},\\
448 M_6\left(\frac{\ell^{2m}n+1}{24} \right)\equiv p\left(\frac{\ell^{2m}n+1}{24} \right)\pmod{\ell^{2m}},\\
5760\mu_4\left(\frac{\ell^{2m}n+1}{24} \right)\equiv -17p\left(\frac{\ell^{2m}n+1}{24} \right)\pmod{\ell^{2m}},\\
967680\mu_6\left(\frac{\ell^{2m}n+1}{24} \right)\equiv 367p\left(\frac{\ell^{2m}n+1}{24} \right)\pmod{\ell^{2m}}.
\end{gather*}
Using Theorem \ref{thm: main-coeff-cong-ak-bd} we can improve these to congruences modulo $ \ell^{5m},$ proving Theorem \ref{thm: crank-congruences}.

\begin{proof}[Proof of Theorem \ref{thm: crank-congruences}]
The proof for each of $ M_4,\,M_6,\,\mu_4, $ and $ \mu_6 $ is similar; we will give the proof in the case of $ \mu_6 $ \eqref{eqn: cong-mu6}. Substituting $ n\mapsto \frac{\ell^{2m}n+1}{24} $ into the identity for $ \mu_6(n) $ found in \cite[(8.38)]{MR4490457} we find that
\begin{align*}
\mu_6\left(\frac{\ell^{2m}n+1}{24} \right)&=\frac{-11}{272160}a_{6,1}(\ell^{2m}n)+\left(\frac{-1}{4032}-\frac{\ell^{2m}n+1}{80640}  \right)a_{4,1}(\ell^{2m}n)\\
&+
\left(\frac{157}{544320} -\frac{\ell^{2m}n+1}{103680}+\frac{(\ell^{2m}n+1)^2}{10368}+\frac{(\ell^{2m}n+1)^3}{248832}\right)p\left(\frac{\ell^{2m}n+1}{24} \right).
\end{align*}
The least common multiple of the denominators is $ 8709120; $ moreover, if $ \left(\frac{-n}{\ell} \right)=1 $ then 
\[ a_{4,1}(\ell^{2m}n)\equiv a_{6,1}(\ell^{2m}n)\equiv 0\pmod{\ell^{5m}} \] 
by Theorem \ref{thm: main-coeff-cong-ak-bd}. Hence
\begin{align*}
8709120\mu_6\left(\frac{\ell^{2m}n+1}{24} \right)\equiv (3303+1701\ell^{2m}n+945\ell^{4m}n^2)p\left(\frac{\ell^{2m}n+1}{24} \right)\pmod{\ell^{5m}}.
\end{align*}
    Equation \eqref{eqn: cong-mu6} follows, since $ (8709120,367,189,105)=9$.\par 
    Equations \eqref{eqn:cong-M4}, \eqref{eqn:cong-M6}, and \eqref{eqn:cong-mu4} 
     can be proved with the same method, using \cite[(8.24),(8.25),(8.37)]{MR4490457} and Theorem \ref{thm: main-coeff-cong-ak-bd}.
\end{proof}

\subsection{4th and 6th Moments of Ranks}\label{sec: rank-moment}
The case for the rank moment functions $ N_{2j} $ is complicated by the fact that the rank moment generating function $ R_{2j}(z)=\sum N_{2j}(n)q^n $  is not a quasimodular form; instead, we have
\begin{align*}
q^{-1/24}\eta R_{2j}\in \mathbb{C}[E_2,E_4,E_6]\oplus_{i=0}^{j-1}\Q\Theta^i (R_2),
\end{align*}
 where $ \Theta :=q \frac{d}{dq} $ \cite[Lemma 8.5]{MR4490457}. In this section we will derive congruences for the moments of ranks using the method of Section~\ref{sec: crank-moment} after proving a congruence modulo $ \ell^m $  for $ N_2. $ 

 When Andrews \cite[Theorem 3]{MR2456627} introduced the $ \spt $ function, he found the following identity:
 \begin{align}\label{eqn:N2-spt-ident}
N_2(n)=2np(n)-2\spt(n).
 \end{align}
Work of Bringmann \cite{MR2437679} shows that $ \spt $ is related to a mock modular form of weight $ 3/2 $. Recall that $ F:\H\to\C $ is a harmonic weak Maass form of weight $ k $ and multiplier $ \nu $  on $ \slz $ if it is a real-analytic function satisfying
\begin{align*}
F \left(\abcd z\right) = \nu \left(\abcd\right)(cz+d)^{k}F(z)
\end{align*}
for all $ \abcd\in\slz; $ is annihilated by the hyperbolic Laplacian $ \Delta_k $; and has linear exponential growth at the cusp $ \infty $ \cite[Definition 15.3.2]{MR3675870}. When $ F $ satisfies these conditions we write $ F\in H_{k}(1,\nu). $ Harmonic weak Maass forms can be written as the sum of a weakly holomorphic function and a period integral of a cusp form. See, for example, \cite[Lemma 5.17]{MR3729259}. 
One can adapt \cite{MR2437679}, as in \cite[\S 2]{MR3788640}, to obtain the following.
\begin{theorem}
Define $ H:\H\to \C $ by
\begin{align}\label{eqn:H-spt}
H(z):= \frac{1}{q^{1/24}} - \sum_{\begin{subarray}{c}
  n\geq 23
\end{subarray}}
\left[12\spt \left(\frac{n+1}{24} \right)+n p \left(\frac{n+1}{24} \right)\right]q^{n/24} := \sum_{n\geq -1}h(n)q^{n/24}.
\end{align}
 Then $ H $ can be completed to a harmonic weak Maass form $ \tilde{H}=H+N \in H_{3/2}(1,\nu_\eta^{-1})$, where $ N(z) $ is a nonholomorphic function given by a period integral of the cusp form $ \eta. $ 
\end{theorem}

In what follows let $ D>0 $ be such that $ D\equiv 1\pmod{24} $ and $ \tilde{H}_D\in H_{3/2}(1,\nu_\eta^{-1}) $ be the unique harmonic weak Maass form such that the holomorphic part of $ \tilde{H}_D $, denoted $ H_D $, is 
\begin{align}\label{eqn:H_D}
  H_D(z) = q^{-D}+O(1) = \sum h_D(n)q^{n/24}.
\end{align}
Such a form exists by \cite[Theorem 3.1]{MR3228930}. Note $ H_1 = H $ in \eqref{eqn:H-spt}. Using the fact that there are no nonzero harmonic weak Maass forms with weakly holomorphic part equal to 0, Ahlgren and Kim \cite[Corollary 3.2]{MR3228930} obtained a family of Hecke relations for weight $ 3/2 $ harmonic Maass forms by applying Hecke operators on $ H_{3/2}(1,\nu_\eta^{-1}) $ and identifying principal parts of their weakly holomorphic terms. Their result was originally stated for $ H_{3/2}\left(576,\left(\frac{12}{\cdot} \right)\right) $, however, it is equivalent to the statement below.

\begin{theorem}\label{thm: Ahlgren-Kim-Hecke-reln}
   Let $ T(\ell^{2m}) $ be the $ \ell^{2m}$th Hecke operator on $ H_{3/2}(1,\nu_\eta^{-1}) $, $ H $ be as in \eqref{eqn:H-spt}, and $ H_D $ as in \eqref{eqn:H_D}. If $ \ell\geq 5 $ is prime then
  \begin{align*}
  H \big| T(\ell^{2m})-\left(\frac{12}{\ell} \right)H \big|T(\ell^{2m-2})=\ell^{m}H_{\ell^{2m}}.
  \end{align*}

\end{theorem}

A corollary implicit in Theorem \ref{thm: Ahlgren-Kim-Hecke-reln} is a relation between the coefficients of $ H $ and $ H_{\ell^{2m}}. $  
\begin{corollary}\label{cor: Hecke-reln-H-spt}
  With the same conditions as Theorem \ref{thm: Ahlgren-Kim-Hecke-reln}, if $ \left(\frac{-n}{\ell} \right)=1 $ and $ m\geq 1 $ then
  \begin{align}\label{eqn:HMF-hl2m}
  h(\ell^{2m}n)=\ell^m h_{\ell^{2m}}(n).
  \end{align}
\end{corollary}
This can be proved by using Theorem \ref{thm: Ahlgren-Kim-Hecke-reln} with \cite[Theorem 3.1]{MR3228930} to prove analogues of Lemmas \ref{lem:tech-1} and \ref{lem:tech-3}. Note, in particular, that the congruence $ h(\ell^{2m}n)\equiv 0\pmod{\ell^m} $ is proved in \cite[Theorem 1.1]{MR2822242}.
\begin{proof}[Proof Sketch.]
For $ \ell, m $ as above write
\begin{gather*}
H_\ell^{(m)}:= H \big| T(\ell^{2m})-\left(\frac{12}{\ell} \right)H \big|T(\ell^{2m-2}),
\end{gather*}
    where $ T(\ell^2) $ acts on the weakly holomorphic parts of forms in $ H_{3/2}(1,\nu_\eta^{-1}) $ as in \eqref{eqn: hecke-op-defn} with $ k=2 $ and $ s=1 $. Denote the weakly holomorphic part of $ H_\ell^{(m)} $ by 
   \[\sum h_\ell^{(m)}(n)q^{n/24}.\]
Using the same methods as in the proofs of Theorem \ref{thm:tech-main} and Lemma \ref{lem:tech-1} one can show
\begin{gather}
  \label{eqn:hl1}
    h_{\ell}^{(1)}(n) = h(\ell^2 n)+ \left(\frac{12}{\ell} \right)\left[\left(\frac{-n}{\ell} \right)-1\right]h(n)+\ell h \left(\frac{n}{\ell^2} \right),\\
  \label{eqn:hl2m}
h_\ell^{(m)}(\ell^2 n)-\ell h_\ell^{(m-1)}(n)=h(\ell^{2m+2}n)-\left(\frac{12}{\ell} \right)h(\ell^{2m}n).
\end{gather}
In the process of proving this claim one obtains, for $ m\geq 2, $  the identity 
\begin{gather}\label{eqn:hlm-recur}
H_\ell^{(m)} = H_\ell^{(m-1)}\big|T(\ell^2)-\ell H_\ell^{(m-2)}.
\end{gather}
When $ \ell\nmid n $ equation \eqref{eqn:hl1} is the base case $ m=1 $ for the claim that
\begin{align*}
h_\ell^{(m)}(n)=h(\ell^{2m}n)+\left(\frac{12}{\ell} \right)\left[\left(\frac{-n}{\ell} \right)-1\right]\sum_{j=1}^m \left(\frac{-12n}{\ell} \right)^j h_\ell(\ell^{2m-2j}n),
\end{align*}
an analogue of Lemma \ref{lem:tech-3} for $ H_{3/2}(1,\nu_\eta^{-1}) $. The remaining cases $ m\geq 2 $ are proven inductively using \eqref{eqn:hl2m} and \eqref{eqn:hlm-recur}. Corollary \ref{cor: Hecke-reln-H-spt} then follows from Theorem \ref{thm: Ahlgren-Kim-Hecke-reln}.
\end{proof}

Corollary \ref{cor: Hecke-reln-H-spt}, \eqref{eqn:N2-spt-ident}, and \eqref{eqn:H-spt} together give a Hecke relation for $ N_2 $.

\begin{corollary}\label{cor:N2-ident}
If $ \ell\geq 5 $ is prime, $ m\geq 1, $ and $ \left(\frac{-n}{\ell} \right)=1 $ then
\begin{align*}
N_2 \left(\frac{\ell^{2m}n+1}{24} \right)=\left(\frac{1}{12} +\frac{\ell^{2m}n}{4} \right)p \left(\frac{\ell^{2m}n+1}{24} \right)+\frac{\ell^m}{6} h_{\ell^{2m}}(n).
\end{align*}
  Consequently, for these $ m,n,\ell $ we have
  \begin{align*}
  12N_2\left(\frac{\ell^{2m}n+1}{24} \right)\equiv p\left(\frac{\ell^{2m}n+1}{24} \right)\pmod{\ell^{m}}.
  \end{align*}
  
\end{corollary}
Note that the congruence above also follows from \eqref{eqn:N2-spt-ident} and \cite[Theorem 1.1]{MR2822242}.
Using Theorem \ref{thm: main-ak-cong} and Corollary \ref{cor:N2-ident}, together with the identities \cite[(8.29), (8.30), (8.45), (8.46)]{MR4490457}, we obtain the following congruences for moments of ranks.

\begin{theorem}\label{thm:rank-congruences}
If $ \ell\geq 5 $ is a prime, $ m\geq 1, $ and $ \left(\frac{-n}{\ell} \right)=1 $ then
\begin{gather}
80 N_4\left(\frac{\ell^{2m}n+1}{24} \right)\equiv 448N_6\left(\frac{\ell^{2m}n+1}{24} \right)\equiv p\left(\frac{\ell^{2m}n+1}{24} \right)\pmod{\ell^m},\\
5760\eta_4\left(\frac{\ell^{2m}n+1}{24} \right)\equiv -17p\left(\frac{\ell^{2m}n+1}{24} \right)\pmod{\ell^m},\\
967680\eta_6\left(\frac{\ell^{2m}n+1}{24} \right)\equiv 367p\left(\frac{\ell^{2m}n+1}{24} \right)\pmod{\ell^m}.
\end{gather}
  
\end{theorem}

The proof proceeds along the same lines as the proof for the $ \mu_6 $ congruence in Theorem \ref{thm: crank-congruences}, presented in Section~\ref{sec: crank-moment}.

\subsection{Proof of Theorem \ref{thm: spt-main}}\label{sec: spt-proof}

We first observe that for $ \ell\geq 5 $ prime and $ \left(\tfrac{-n}{\ell} \right)=1 $ we have from Theorems \ref{thm: crank-congruences} and \ref{thm:rank-congruences}
\begin{align*}
  5760\mu_4 \left(\frac{\ell^{2m}n+1}{24} \right)\equiv 5760\eta_4 \left(\frac{\ell^{2m}n+1}{24} \right)\equiv -17 p \left(\frac{\ell^{2m}n+1}{24} \right)\pmod{\ell^{m}},\\
  967680\mu_6 \left(\frac{\ell^{2m}n+1}{24} \right)\equiv 967680\eta_6 \left(\frac{\ell^{2m}n+1}{24} \right)\equiv 367 p \left(\frac{\ell^{2m}n+1}{24} \right)\pmod{\ell^{m}}.
\end{align*}
Then from \eqref{eqn:spt-defn} we have
\begin{gather*}
  \spt_2 \left(\frac{\ell^{2m}n+1}{24} \right)\equiv 0\pmod{\ell^{m-\delta_{5,\ell}}},\\
  \spt_3 \left(\frac{\ell^{2m}n+1}{24} \right)\equiv 0\pmod{\ell^{m-\delta_{5,\ell}-\delta_{7,\ell}}}.
\end{gather*}

To show $ \spt_2 \left(\frac{5^{2m}n+1}{24} \right)\equiv 0\pmod{5^{m}} $ for $ \left(\frac{-n}{5} \right)=1, $ note that from \cite[(8.64)]{MR4490457} and Corollary \ref{cor:N2-ident} we have
\begin{multline*}
  \spt_2 \left(\frac{5^{2m}n+1}{24} \right)=-\frac{a_{4,1}(5^{2m}n)}{288}+h_{5^{2m}}(n) \left(\frac{5^{3 m}n}{288} +\frac{5^m}{288}\right)\\
  +p \left(\frac{5^{2m}n+1}{24} \right) \left(\frac{ 5^{4 m}n^2}{144} +\frac{5^{2 m}n}{288} -\frac{1}{90}\right).
  \end{multline*}
 The only term with $ 5 $ in the denominator is $ -\frac{p \left(\frac{\ell^{2m}n+1}{24} \right)}{90}.  $ But $ p \left(\frac{5^{2m}n+1}{24} \right)\equiv 0\pmod{5^{2m}} $ for $ m\geq 1 $ by the Ramanujan congruences, finishing the proof of Theorem \ref{thm: spt-main} for $ \spt_2. $\par 

 Now for $ \spt_3\left(\frac{\ell^{2m}n+1}{24} \right). $ For $ \left(\frac{-n}{\ell} \right)=1 $ we have
 \begin{multline}
 \spt_3 \left(\frac{\ell^{2m}n+1}{24} \right) =\left(\frac{n \ell ^{2 m}}{11520}+\frac{1}{2304}\right)a_{4,1}(\ell^{2m}n)+\frac{1}{38880}a_{6,1}(\ell^{2m}n)\\
 +\left(-\frac{n^2 \ell ^{5
  m}}{23040}-\frac{n \ell ^{3 m}}{2304}-\frac{\ell ^m}{2560}\right)h_{\ell^{2m}}(n)+ \left(-\frac{13 n^3 \ell ^{6 m}}{124416}-\frac{n^2 \ell ^{4
  m}}{1152}-\frac{n \ell ^{2 m}}{2560}\right)p\left(\frac{\ell^{2m}n+1}{24}\right).
 \end{multline}
 All terms are divisible by powers of $ \ell $ greater than $ m+1 $ except $ \frac{-\ell^m}{2560}h_{\ell^{2m}}.$ The observation that $ 7\nmid 2560 $ but $ 5\mid\mid 2560 $ concludes the proof of Theorem \ref{thm: spt-main} for $ \spt_3. $ \par

We now expand the generating functions for $ \spt_4,\,\spt_5 $ as elements in $ \mathbb{C}[E_2,E_4,E_6]\oplus_{i=0}^{j-1}\Q\Theta^i (R_2) $ \cite[Lemma 8.5]{MR4490457}. Doing so with Mathematica \cite{Mathematica}, using \cite[(1.3)]{MR2822233}, yields the following lemma.
 
\begin{lemma}\label{lem:spt45}
For all $ n\geq 0 $ we have
\begin{multline}\label{eqn:spt4-modular}
  \spt_4(n)=-\mfrac{67}{191600640}a_{8,1}(24n-1)+\left(-\mfrac{43
  n}{2993760}-\mfrac{19}{3991680}\right)a_{6,1}(24n-1)\\
  +\left(-\mfrac{23 n^2}{60480}-\mfrac{73
  n}{181440}-\mfrac{431}{8709120}\right)a_{4,1}(24n-1)+\left(\mfrac{3
  n^3}{140}+\mfrac{n^2}{35}+\mfrac{n}{140}\right)N_2(n)\\
  +\left(\mfrac{59
  n^4}{630}+\mfrac{121 n^3}{1260}+\mfrac{223 n^2}{60480}-\mfrac{1271
  n}{544320}+\mfrac{317}{5806080}\right) p(n),
\end{multline}

\begin{multline}\label{eqn:spt5-modular}
  \spt_5(n)=\mfrac{551}{93405312000}a_{10,1}(24n-1)+\left(\mfrac{2831
  n}{14370048000}+\mfrac{4043}{43110144000}\right)a_{8,1}(24n-1)\\
  +\left(\mfrac{151
  n^2}{77837760}+\mfrac{19
  n}{4717440}+\mfrac{3281}{3736212480}\right)a_{6,1}(24n-1)\\
  + \left(\mfrac{71
  n^3}{1995840}+\mfrac{85 n^2}{798336}+\mfrac{485
  n}{6386688}+\mfrac{355}{43110144}\right)a_{4,1}(24n-1)\\
  +\left(-\mfrac{n^4}{560}-\mfrac
  {n^3}{168}-\mfrac{3 n^2}{560}-\mfrac{n}{840}\right) N_2(n)\\
  +\left(-\mfrac{107
  n^5}{12600}-\mfrac{8 n^4}{315}-\mfrac{349 n^3}{20160}-\mfrac{61
  n^2}{155520}+\mfrac{751 n}{1935360}-\mfrac{2407}{261273600}\right) p(n).
\end{multline}

\end{lemma}

The remaining claims of Theorem \ref{thm: spt-main} follow from this lemma using an analysis similar to that for $ \spt_2 $ and $ \spt_3 $  above. 

We demonstrate this in the case of $ \spt_4 $. Let $ \ell\geq 5 $ be a prime and $ \left(\tfrac{-n}{\ell} \right)=1 $. Substituting $ n\mapsto \tfrac{\ell^{2m}n+1}{24}  $ in \eqref{eqn:spt4-modular} yields
\begin{multline}
  \spt_4 \left(\frac{\ell^{2m}n+1}{24}\right) = -\frac{67}{191600640}a_{8,1}(\ell^{2m}n)+\left(-\frac{43 \ell^{2
    m}n}{71850240}-\frac{1}{186624}\right)a_{6,1}(\ell^{2m}n)+\\
     \left(-\frac{23  \ell^{4 m}n^2}{34836480}-\frac{\ell^{2 m}n}{55296}-\frac{37}{552960}\right)a_{4,1}(\ell^{2m}n)+\left(\frac{ \ell^{7 m}n^3}{3870720}+\frac{\ell^{5
    m}n^2}{110592}+\frac{37\ell^{3 m} n }{552960}+\frac{5 \ell^m}{86016}\right)h_{\ell^{2m}}(n)\\
    +\left(\frac{\ell^{8 m}n^4 }{1492992}+\frac{65  \ell^{6
    m}n^3}{2985984}+\frac{37 \ell^{4 m}n^2 }{276480}+\frac{5\ell^{2 m}n}{86016}\right)p\left(\frac{\ell^{2m}n+1}{24}\right).
  \end{multline}
The least common multiple of the denominators is $ 1149603840 = 2^{12}\cdot 3^{6}\cdot 5\cdot 7\cdot 11; $ however, $ a_{k,1}(\ell^{2m}n)\equiv 0\pmod{\ell^{5m}} $ for $ k=4,6,8 $ and $ \frac{5\ell^{2 m}n}{86016}p\left(\frac{\ell^{2m}n+1}{24}\right) \equiv 0\pmod{\ell^{3m+1}}. $ Therefore we need only consider the coefficient of $ h_{\ell^{2m}} $: because the denominator of $ \frac{5 \ell^m}{86016} $ is divisible by $ 7 $ we lose a power of $ 7 $ but gain a power of $ 5 $ in the numerator.\par 

In the case of $\spt_5\left(\frac{\ell^{2m}n+1}{24} \right) $ each term in the quasimodular expansion is divisible by powers of $ \ell $ larger than $ m+1 $  except $ \frac{-35}{3538944}\ell^{m}h_{\ell^{2m}}(n), $ hence the higher-power congruence for $ \spt_5 $ when $ \ell=5 $ or $ 7. $

\bibliographystyle{alpha}
\bibliography{hecke-rln.bib}
\end{document}